\documentclass[11pt,reqno]{amsart}
\usepackage[a4paper,scale=0.80,twoside=false]{geometry} 

\usepackage[english]{babel}
\usepackage{microtype}
\usepackage{amsfonts}
\usepackage[utf8]{inputenc}	
\usepackage[T1]{fontenc}   
\usepackage{amsmath}
\usepackage{xfrac}
\usepackage{graphicx}
\usepackage[colorlinks=true, allcolors=black]{hyperref}
\usepackage{enumitem}
\usepackage{tikz}
\usepackage{amssymb}
\usepackage{amsthm}
\usepackage{color}
\usepackage{subcaption}
\usepackage{afterpage}
\usepackage{mathtools}
\usepackage{pdfpages}
\usepackage{bm}
\usepackage{appendix}
\usepackage{cleveref}
\usepackage{tikz}
\usepackage{comment}
\usepackage{float}
\usepackage{empheq}
\usepackage{accents}

\usepackage[style=ext-numeric-comp,
giveninits=true,
articlein=false,
backend=biber]{biblatex}
\DeclareSourcemap{
    \maps[datatype=bibtex]{
        \map{
            \step[fieldset=issn, null]
        }
    }
}
\usepackage{csquotes}

\newtheorem{theorem}{Theorem}[section]
\newtheorem{lemma}[theorem]{Lemma}
\newtheorem{proposition}[theorem]{Proposition}

\theoremstyle{definition}
\newtheorem{definition}[theorem]{Definition}
\newtheorem{remark}[theorem]{Remark}

\newcommand{\R}{\mathbb{R}}
\newcommand{\N}{\mathbb{N}}
\newcommand{\Z}{\mathbb{Z}}
\newcommand{\T}{\mathbb{T}}

\newcommand{\hrefemail}[1]{\href{mailto:#1}{#1}}

\newcommand{\ac}[2]{\accentset{#1}{#2}}

\newlist{properties}{enumerate}{10}
\setlist[properties]{label*=(\roman*)}

\crefname{propertiesi}{property}{properties}
\Crefname{propertiesi}{Property}{Properties}

\title{On small-amplitude asymmetric water waves}
\author[D. S. Seth]{Douglas Svensson Seth}

\dedicatory{Department of Mathematical Sciences\\ Norwegian University of Science and Technology\\ 7491 Trondheim, Norway} 
\thanks{E-mail address: \hrefemail{douglas.s.seth@ntnu.no}\\
        \indent{}Orcid: \href{https://orcid.org/0000-0001-7339-5507}{0000-0001-7339-5507}}


\begin{document}
\begin{abstract}
    We generalize the method used by Mæhlen \& Seth \cite{Maehlen_2023} used to prove the existence of small-amplitude asymmetric solutions to the capillary-gravity Whiham equation, so that it can be applied directly to a class of similar equations. The purpose is to prove or disprove the existence of asymmetric waves for the water wave problem or other model equations for water waves. Our main result in this paper is a theorem that gives both necessary and sufficient conditions for the existence of small-amplitude periodic asymmetry solutions for this class of equations. The result is then applied to an infinite depth capillary-gravity Whitham equation and an infinite depth capillary-gravity Babenko equation to show the nonexistence of small-amplitude waves for these equations. This example also highlights the similarities between these equations suggesting the potential existence of small-amplitude asymmetric waves for the finite depth capillary-gravity Babenko equation.
\end{abstract}
\maketitle

\section{Introduction}
\subsection*{Background} 
Most of the theory on traveling water waves is developed under the assumption of symmetry around a vertical axis; see \cite{Haziot_2022} for an overview. This symmetry combined with the translational invariance of the problem can be turned into an assumption of even waves. The theory of symmetric traveling waves is rich but also obviously incomplete. It says nothing about what happens if we drop the symmetry assumption and consider asymmetric (or non-symmetric) waves as well; see \Cref{def:asymmetric} below. That is, to complete the picture it is required to either show the existence of all such asymmetric waves or prove that they cannot exist in various regimes.

A common approach to showing the existence of water waves is through bifurcation. It is well established that symmetry breaking can occur spontaneously through bifurcation, although the equations themselves retain the symmetry in question. Examples of this include Hopf bifurcation breaking temporal symmetry or Bénard convection breaking spatial symmetry \cite{Sattinger_1980}. The water wave problem and many of the derived models (see, e.g. \cite{Lannes2013}) have the aforementioned symmetry, and so do the equations we are dealing with in the present paper. However, we show that this symmetry can be broken through bifurcation under certain conditions.

For pure gravity traveling waves, the question of the existence of asymmetric waves was answered in the case of solitary waves by Craig \& Sternberg \cite{Craig_1988}, who proved that all such waves are necessarily symmetric. For the periodic case, the investigation starts with Chen \& Saffman \cite{Chen_1980}, who found new families of bifurcating solutions on deep water when attempting to compute asymmetric traveling gravity waves. However, these turned out to be even solutions that were shifted, i.e. symmetric. The first result to show any form of the existence of asymmetric periodic gravity waves instead comes from Zufiria \cite{Zufiria_1987WeaklyGravFinD}, who derived a weakly nonlinear Hamiltonian model and showed the existence of asymmetric solutions through symmetry-breaking bifurcation. Zufiria also numerically computed asymmetric water waves using the full Euler equations on infinite depth \cite{Zufiria_1987GravityInfD}.

Also in the case of capillary-gravity traveling waves, one of the first results was obtained by Zufiria \cite{Zufiria_1987GravCapFinD} showing the existence of asymmetric solutions to a weakly nonlinear Hamiltonian model. This result has later been amended by additional numerical studies by Shimizu \& Shoji \cite{Shimizu_2012} and Gao, Wang \& Vanden-Broeck \cite{Gao_2016Investigations,Gao_2016Solitary}.

All these results share the common factor that the symmetry-breaking bifurcation is a secondary bifurcation, i.e. the asymmetric waves bifurcate from a family of nonzero symmetric waves. In contrast to this, Mæhlen \& Seth \cite{Maehlen_2023} showed the existence of a symmetry-breaking bifurcation at the trivial solution to the capillary-gravity Whitham equation. A shallow water wave model proposed by Whitham \cite{Whitham1967} for gravity waves that was later extended to the capillary-gravity case \cite{Lannes_2013}. This can be seen as the result that completes the picture of an earlier result by Ehrnström, Johnson, Maehlen \& Remonato \cite{Ehrnstrom_2019} characterizing all small amplitude symmetric traveling wave solutions. Another avenue that requires examination is the existence of asymmetric waves in the presence of vorticity; see, e.g. Constantin \& Escher \cite{Constantin2004} and Ehrnström, Holden \& Raynaud \cite{Ehrnstrom_2009}.

The nonexistence for asymmetric periodic capillary-gravity waves bifurcating from the trivial solution is proven in \cite[Theorem 4.5]{Okamoto_2001}. However, the model used there seems to reduce the number of free parameters more than necessary in the finite depth case. In the present paper, we will generalize the method used in \cite{Maehlen_2023} so that it can be applied directly to a class of similar equations. This allows us to study the full water wave problem, other model equations for water waves, as well as unrelated equations of similar form with this method, which will help us find or disprove asymmetric waves. We present this as a theorem (\Cref{thm:necessaryandsufficient}) stating necessary and sufficient conditions for the existence of small-amplitude periodic asymmetric solutions for this class of equations. The result is then applied to an infinite depth capillary-gravity Whitham equation to show the nonexistence of small-amplitude asymmetric waves for this equation. The same is done for the capillary-gravity water-wave problem in Babenko's formulation\cite{Buffoni_2000}. This reveals the similarities the equations exhibit, which suggests that one could show the existence of small-amplitude asymmetric waves for the water wave problem with finite depth. 

\subsection*{Method}Since the capillary-gravity Whitham equation served as a blueprint for the class of equations we consider in this paper, we introduce the equation here.
\begin{equation}\label{eq:Whitham}
u_t +(M_{T}u+u^2)_x=0,
\end{equation}
where $M_{T}$ is a spatial Fourier multiplier
\[
\widehat{M_{T}u}(t,\xi)=m_T(\xi)\widehat{u}(\xi,t),
\]
where we use the Fourier transform $\widehat{u}(t, \xi) = \int_{\R} u(t, x)e^{-i\xi x}dx$ and the symbol $m_{T}$ is given by
\[
m_T(\xi)=\sqrt{\frac{(1+T\xi^2)\tanh(\xi)}{\xi}}.
\]
The asymmetric solutions to \cref{eq:Whitham} in \cite{Maehlen_2023} were constructed under the assumption of traveling waves, that is, the assumption that $u(x,t)$ has a fixed shape traveling with speed $c$, or $u=u(x-ct)$. The solution was also assumed to be periodic with period $\frac{2\pi}{\kappa}$, but rescaled to $2\pi$-periodic by changing variables. The equation resulting from these assumptions can be integrated, giving a
\begin{equation}\label{eq:WhitamRed}
-cu+M_{T,\kappa}u+u^2=0,
\end{equation}
where the integration constant is (without loss of generality) set to $0$, and $M_{T,\kappa}$ is a Fourier multiplier with symbol $m_{T,\kappa}(\xi)=m_T(\kappa\xi)$. The paper by Mæhlen \& Seth \cite{Maehlen_2023} focuses almost exclusively on \cref{eq:WhitamRed}. The purpose of this paper is to generalize the result to a class of similar equations. 

In other words:
The main purpose of this paper is to identify necessary and sufficient conditions to apply the method of finding asymmetric waves used in \cite{Maehlen_2023} for a class of equations of a similar form to the capillary-gravity Whitham equation.

The class of equations can be thought of as what we obtain if we let $M_{T,\kappa}$ in \cref{eq:Whitham} be a general Fourier multiplier and replace the $u^2$-term by a general nonlinearity, however, both are still subject to some conditions that we specify below.

We approach the equations with a method that is similar to the classical Crandall \& Rabinowitz bifurcation theorem \cite{Crandall_1971}. We begin by preforming a Lyapunov--Schmidt reduction, so we can write the solution $u=v+w$, where $w$ lies in an infinite-dimensional function space and $v$ in a finite-dimensional function space (the kernel of the linearization). Moreover, the equation itself is decomposed into a finite-dimensional and infinite-dimensional part. The infinite-dimensional part is directly solved with the implicit function theorem giving $w=w(v)$. 

For the finite-dimensional part of the problem, the method really diverges from the classical result. It can be shown that we need $v$ to be asymmetric for the solution to be asymmetric. Since we consider periodic functions, $v$ lie in a space spanned by a finite number of Fourier modes. Due to the symmetry of the equation, the modes always come in sine-cosine pairs with the same wavenumber. The evenness assumption that is common for water waves reduces this to solely cosine functions, but in this paper we avoid making this assumption. Moreover, one such pair is not sufficient for $v$ to be asymmetric because any sum of such a sine-cosine pair can always be shifted to only a cosine through a spatial translation. Thus, we require at least two such pairs, or in other words, that the kernel of the linearization of the equation be four-dimensional. This higher dimensional kernel makes the necessary computations more involved in itself, but the main technical issue is that the elements of the kernel are resonant. Even if $v$ does not contain all the Fourier modes in the kernel, the solution can still contain all the modes, since the existing modes may combine to create the missing ones. To handle these resonances, we need to expand $w$ to the order in which these modes appear, which in general can be arbitrarily high. Fortunately, the structure of the equations allow us to reduce the finite-dimensional problem from four to three dimensions, which means that we only have to handle one such resonance instead of two.

\subsection*{Structure of the paper}The class of equations with which we work in this paper is defined in \Cref{sec:prel}. Additionally, there we collect some of the more elementary results we need for the main result. The main result is contained and proved in \Cref{sec:main}: \Cref{thm:necessaryandsufficient} gives necessary and sufficient conditions for the existence of asymmetric solutions. In the last part, \Cref{sec:example}, we apply the result to the infinite depth Whitham equation to show the nonexistence of small-amplitude asymmetric waves. We also apply the result to the infite depth Babenko equation showing the same type of result, as well as demonstrating the similarities between the two equations. This gives hope for finding asymmetric solutions to the Babenko equation on finite depth and thus asymmetric small-amplitude water waves.
\section{Preliminaries}\label{sec:prel}
\subsection*{Asymmetry and function spaces}We begin by giving a stricter definition of symmetric and asymmetric functions than the one given in the introduction. The reason for this definition of symmetry and asymmetry is the translational invariance of the equation. Any vertical line of symmetry is equivalent to evenness; conversely, for a function to be asymmetric, there cannot be any vertical line of symmetry.
\begin{definition}[Symmetric and asymmetric functions]\label{def:asymmetric} 
    We say that a function $u:\mathbb{R}\to \R$ is symmetric if there exists an $a\in\R$ such that $x\mapsto u(x+a)$ is an even function, and if no such $a$ exists then we say that $u$ is asymmetric.
\end{definition}

With this definition in hand, we turn our attention to function spaces. The theory works for a quite general choice of function spaces, but one can imagine some of the common spaces used in the study of differential equations such as Hölder spaces, Sobolev spaces, etc. Naturally, we want our function spaces to be Banach spaces. The other two important properties are that they consist of functions on the unit circle $\T=\R/2\pi\Z$ (or equivalently that they consist of $2\pi$-periodic functions on $\R$) and that we can express the elements as Fourier series. Thus, we make the following definition.
\begin{definition}[Function spaces] Let $X^s(\T)$ denote a scale of Banach spaces that consists of functions defined on the unit circle $\T$ such that $X^s(\T)\hookrightarrow L^2(\T)$ for all $s\in S\subset \R$.
\end{definition}
 We only concern ourselves with spaces with functions defined on the unit circle because functions of any other period can always be rescaled to $2\pi$-periodic functions. In fact, this is a crucial step in \cite{Maehlen_2023}, since the period is needed as a bifurcation parameter. Moreover, since all functions we consider can be written as Fourier series, we often define Fourier multipliers through the basis $\{e^{ikx}\}_{k\in \mathbb{Z}}$. For these functions, defining a multiplier $M$ that acts through $\widehat{M u}(\xi)=m(\xi)\hat{u}(\xi)$ is equivalent to defining its action through $M e^{ikx}=m(k)e^{ikx}$ (for all $k$). 

 In the applications in \Cref{sec:example}, we use Zygmund spaces as function spaces for the equation. These are defined as follows. Let $\varphi$ be the Fourier inverse of a real-valued symmetric function $\hat{\varphi}\in C_c^\infty(\R)$ that satisfies $\hat{\varphi}(\xi)=1$ for $|\xi|\leq 1$ and $\hat{\varphi}(\xi)=0$ for $|\xi|\geq 2$. Then set $\phi_0\coloneqq \varphi$ and $\phi_j(x)\coloneqq 2^j\varphi(2^jx) - 2^{j-1}\varphi(2^{j-1}x)$ for $j\in \N$. The Zygmund space $\mathcal{C}^{s}(\T)$ consists of real-valued distributions $u$ on $\T$ such that
\begin{align*}
    \|u\|_{\mathcal{C}^{s}(\T)}\coloneqq \sup_{j} 2^{js}\|u\ast \phi_j\|_{L^\infty(\T)}<\infty.
\end{align*}
For all $s'>s\geq 0$, the space $\mathcal{C}^{s'}(\T)$ is compactly embedded in $\mathcal{C}^s(\T)$. 
Moreover, the Hölder space $C^s(\T)$ naturally embeds in $\mathcal{C}^s(\T)$ for all $s\geq 0$, while the converse is true for non-integer $s$; see, e.g.\cite[Section 13.8]{Taylor_1996} for a comprehensive treatment.
 
 Finally, we denote a ball of radius $r$ centered at the point $p$ in the space (or subset of space) $X$ by $B_r^X(p)$. A claim containing a ball with radius $\varepsilon$ means that there exists a sufficiently small radius for the claim to be satisfied. 

\subsection*{The class of equations}The equations we consider are of the form
\begin{equation}\label{eq:fundamentalequation}
    F(\mu,u)=L(\mu)u+N(\mu,u)=0,
\end{equation}
where $\mu=(\mu_1,\ldots,\mu_n)\in M\subset\R^n$ is a set of parameters and $u$ is a function in $X^s(\T)$, $s,t\in S$. We also assume that the equation satisfies the following properties.
\begin{properties}[topsep=14pt, itemsep=14pt]
    \item\label{property1} The linear part of the equation 
    \[
        L(\mu):X^s(\T)\to X^t(\T)
    \]
    is a Fredholm operator of index 0 for all $\mu\in M$.
    Moreover, the linear part of the equation is a Fourier multiplier, so the action on a single Fourier mode is 
    \[
        L(\mu)e^{ikx}=l_\mu(k)e^{ikx},
    \]
    and the symbol is even, $l_\mu(-k)=l_\mu(k)$.
    \item\label{property2} The nonlinear part of the equation can be written
    \[
        N(\mu,u)=\sum_{m=2}^\infty N_m(\mu,(u,\ldots,u)),
    \]
    where $N_m(\mu,\cdot):(X^s(\T))^m\to X^t(\T)$, are $m$-linear operators that acts through 
\[
        N_m(\mu,(e^{ik_1x},\ldots,e^{ik_mx}))=n_{m,\mu}(k_1,\ldots,k_m)e^{i(k_1+\ldots+k_m)x}.
\]
Moreover, $n_{m,\mu}(k_1,\ldots,k_m)=n_{m,\mu}(-k_1,\ldots,-k_m)$, which means that $N$ preserves evenness. That is, if $u(-x)=u(x)$, then $N(\mu,u(-x))=N(\mu,u(x))$.
    \item\label{property3}The equation is variational. That is, there exists a functional $\mathcal{J}_\mu(u)$ such that 
    \[
        D_u\mathcal{J}_\mu(u)w=\int_\T (L(\mu)u+N(\mu,u))w\,dx=\langle L(\mu)u+N(\mu,u),w\rangle.
    \]
    \item\label{property4} The map $F:M\times X^s(\T)\to X^t(\T)$ is analytic in the sense defined in \cite{Buffoni_2004}.
\end{properties}

\subsection*{Lyapunov--Schmidt reduction}\Cref{property1} implies that for any $\mu\in M$ there exists a corresponding finite set $K_\mu$ of integers such that $l_\mu(k)=0$ for all $k\in K_\mu$. We begin by considering $\mu_0$ so that $K_{\mu_0}=\emptyset$. Then $L(\mu_0)$ is a Fredholm operator with $\ker L(\mu_0)=\{0\}$, which means that it is an invertible operator. Thus, $F(\mu_0,0)=0$ and $D_uF(\mu_0,0)=L(\mu_0)$ are invertible. Hence we can apply the implicit function theorem which gives us the existence of a unique solution in $B^{X^{s}(\T)}_{\varepsilon}(0)$ to $F(\mu,u)=0$ for every $\mu\in B^M_\varepsilon(\mu_0)$. However, we clearly have $F(\mu,0)=0$ for all $\mu\in B^M_\varepsilon(\mu_0)$, so the only solution is the trivial one, which is symmetric. Thus, we have the following result.
\begin{lemma}\label{lemma:TrivialKernel}
If $K_{\mu_0}=\{0\}$, then there exists no solution $u\in X^s(\T)$ to \cref{eq:fundamentalequation} for any $\mu\in B^M_\varepsilon(\mu_0)$ except $u\equiv 0$.
\end{lemma}

The above result is obviously well known, but means that we need $\mu_0$ such that $ K_{\mu_0}$ is nonempty to find arbitrarily small asymmetric solutions. Thus, we assume that there is a point $\mu_0\in M$ such that $  K_{\mu_0}$ is nonempty. This implies
\[
\ker L(\mu_0)=V,
\]
where $\dim V=\vert K_{\mu_0}\vert>0$. In fact, $\vert K_{\mu_0}\vert$ is an even number because $l_\mu$ is even, so $\vert K_{\mu_0} \vert\geq 2$. In this case, we can perform a Lyapunov--Schmidt reduction. Since $L$ is invariant on subspaces spanned by a single Fourier mode, we can decompose $X^s(\T)=V\oplus W^s$ and $X^t(\T)=V\oplus W^t$, where $W^t$ is the image of $L(\mu_0)$. We also define the corresponding $L^2$-orthogonal projection $P_V:L^2(\T)\to V$ and $P_W=I-P_V$. These allow us to write $u=P_Vu+P_Wu=v+w$ and \cref{eq:fundamentalequation} as
\begin{align}
    L(\mu)v+P_VN(\mu,v+w)&=0,\label{eq:findim}\\
    L(\mu)w+P_WN(\mu,v+w)&=0.\label{eq:infdim}
\end{align}
\begin{lemma}\label{lemma:LS-reduction}
	For any $v\in B^{V}_\varepsilon(0)$ and $\mu\in B^{M}_\varepsilon(\mu_0)$ there exists a unique function $w\in W^s$ solving \cref{eq:infdim}. Moreover, the mapping $(\mu,v)\mapsto w$ is analytic.
\end{lemma}
\begin{proof}
We apply the implicit function theorem to the operator $F_W:W^s\times V\times M\to W^t$ defined by $(w,v,\mu)\mapsto L(\mu)w+P_WN(\mu,v+w)$. Clearly, $D_wF_W(0,0,\mu_0)=L(\mu_0)$, which is an isomorphism from $W^s$ to $W^t$. The result follows from the analytic implicit function theorem \cite{Buffoni_2016}.
\end{proof}
Now to find a nontrivial solution to \cref{eq:fundamentalequation}, it only remains to solve \cref{eq:findim} for some nontrivial $v$. However, for this to be an asymmetric solution of \cref{eq:fundamentalequation}, it is necessary to solve \cref{eq:findim} for some asymmetric $v$, as shown in the result below.

\begin{lemma}\label{lemma:asymmetryequivalence}Let $v\in B^{V}_\varepsilon(0)$ and $\mu\in B^{M}_\varepsilon(\mu_0)$, and define $u=v+w$, where $w\in W^s$ is the solution to \cref{eq:infdim} from \Cref{lemma:LS-reduction}. Then $u$ is asymmetric if and only if $v$ is asymmetric.
\end{lemma}
\begin{proof}We begin by assuming that $u$ is symmetric. Because we have translational invariance for the equation this is equivalent to assuming that that $u$ is even. Then
$v(-x)+w(-x)=u(-x)=u(x)=v(x)+w(x)$ or $v(-x)-v(x)=w(-x)-w(x)$, which means $v(-x)=v(x)$ and $w(-x)=w(x)$ because $V$ and $W^s$ are orthogonal.

On the other hand, since $N$ maps even functions to even functions, we can restrict the problem to the subspace of $X^s(\T)$ consisting solely of even functions. Then we can find an even $w$ such that \cref{eq:infdim} is solved for any even $v$. This means that if $v$ is symmetric, so is $u$. 

Thus, we have proven $u$ is symmetric if and only if $v$ is symmetric, which is equivalent to the statement of the lemma.
\end{proof}
This result has a direct consequence for the case where $K_{\mu_0}=\{\pm k\}$. Then we find that the kernel is
\[
V=\text{span}\{\cos(kx),\sin(kx)\}.
\]
However, for any $(t_1,t_2)$ there exist $\rho$ and $\phi$ such that
\[
v=t_1\cos(kx)+t_2\sin(kx)=\rho\cos(k(x+\phi)).
\]
In other words, if $K_{\mu_0}=\{\pm k\}$ then any $v\in V$ is symmetric. This gives us the following result.
\begin{lemma}\label{lemma:KernelSize2}
	If $K_{\mu_0}=\{\pm k\}$, then there exists no arbitrarily small asymmetric solution $u\in X^s(\T)$ to \cref{eq:fundamentalequation} for any $\mu\in B^M_\varepsilon(\mu_0)$.
\end{lemma}
\begin{proof} Assume we have a nontrivial solution $u$. If $u$ is sufficiently small, then $P_Vu=v\in B^{V}_\varepsilon(0)$, so there exists a unique solution $w$ to \cref{eq:infdim} according to \cref{lemma:LS-reduction}. However, by the uniqueness we get $P_Wu=w$. Since $K_{\mu_0}=\{\pm k\}$ means that all elements in $V$ are symmetric, we can apply \cref{lemma:asymmetryequivalence} to conclude that $u$ is symmetric.
\end{proof}
In general, we expect \cref{eq:findim} to be a number of scalar equations equal to $\vert K_{\mu_0}\vert$. However, the above result shows that when $\vert K_{\mu_0}\vert=2$ we cannot obtain anything but symmetric solutions, which due to the translational invariance are equivalent to even solutions. Thus, all essentially different solutions can be obtained even if we a priori restrict the problem to even solutions, but then the dimension of $V$ is reduced to one. In particular, this means that \cref{eq:findim} is essentially one scalar equation when $\vert K_{\mu_0}\vert=2$. A similar reduction of the number of equations in \cref{eq:findim} is possible for any $\vert K_{\mu_0}\vert$. Using the variational structure and translational invariance, we can show that \cref{eq:findim} is essentially $\vert K_{\mu_0}\vert-1$ scalar equations.
\begin{lemma}Let $v\in B^{V}_\varepsilon(0)$ and $\mu\in B^{M}_\varepsilon(\mu_0)$ and define $u=v+w$, where $w\in W^s$ is the solution to \cref{eq:infdim} from \Cref{lemma:LS-reduction}. Then
\[
\langle L(\mu)v+P_V N(\mu,v+w),v'\rangle=0.
\]
\end{lemma}
\begin{proof} By the translational invariance we have
\[
0=\partial_a \mathcal{J}_\mu(u(\cdot + a))\vert_{a=0}=\langle L(\mu)u +N(\mu,v+w),w'\rangle+\langle L(\mu)u+N(\mu,v+w),v'\rangle,
\]
where $v'\in V$ and $w'\in W^s$, so
\[
\langle L(\mu)u +N(\mu,v+w),w'\rangle=\langle L(\mu)w +P_W N(\mu,v+w),w'\rangle
\]
and
\[
\langle L(\mu)u+N(\mu,v+w),v'\rangle=\langle L(\mu)v+P_VN(\mu,v+w),v'\rangle.
\]
Because $w$ is a solution to \cref{eq:infdim}, we get $\langle L(\mu)w +P_W N(\mu,v+w),w'\rangle=0$ that proves the result.
\end{proof}
This lemma shows that one of the equations in \cref{eq:findim} is always satisfied. Thus, naively we need $\mu\in \R^{\vert K_{\mu_0}\vert -1}$ to be able to solve \cref{eq:findim}. We will consider this problem for $\vert K_{\mu_0}\vert=4$ in the following section, the lowest size of $ K_{\mu_0}$ for which we can find asymmetric waves.
\section{Sufficient and necessary conditions}\label{sec:main}
\subsection*{Taylor-Fourier expansions} In this section, we focus mainly on the case where $\mu_0$ is such that $K_{\mu_0}=\{\pm k_1,\pm k_2\}$. Then the kernel of the linearization is given by
\[
    V=\text{span}\{\cos(k_1x),\sin(k_1x),\cos(k_2x),\sin(k_2x)\}.
\]
However, following \cite{Maehlen_2023}, we shall write a general element $v\in V$ as
\[
    v=r_1\cos(k_1(x+\theta_1))+r_2\cos(k_1(x+\theta_2)),
\]
determined by the parameters $r=(r_1,r_2)$ and $\theta=(\theta_1,\theta_2)$. We denote multi-indices by $\alpha,\gamma\in\mathbb{N}_0^2$.
Let
\begin{align*}
E=\left(e^{ik_1(x+\theta_1)},e^{ik_2(x+\theta_2)}\right),\\
k_{\alpha,\gamma}=(\alpha_1-\gamma_1)k_1+(\alpha_2-\gamma_2)k_2,
\end{align*}
and recall the convention that $r^\alpha=r_1^{\alpha_1}r_2^{\alpha_2}$. We also define the operator $\mathcal{L}=L^{-1}(\mu)P_W$, which has symbol
\[
\ell(k)=\left\{
\begin{aligned}
&l^{-1}_{\mu}(k) &&\text{ if }k\notin K_{\mu_0},\\
&0 &&\text{ if }k\in K_{\mu_0}.
\end{aligned}
\right.
\]
With this notation in hand, we can show a Taylor-Fourier expansion of the solution.
\begin{lemma}\label{lemma:expansions}Let $v\in B^{V}_\varepsilon(0)$ and $\mu\in B^{M}_\varepsilon(\mu_0)$ and define $u=v+w$, where $w\in W^s$ is the solution to \cref{eq:infdim} from \Cref{lemma:LS-reduction}. Then there exist coefficients $\hat{u}_{\alpha,\gamma}=\hat{u}_{\alpha,\gamma}(\mu)$ such that
\[
    u=\sum_{\vert \alpha\vert+\vert\gamma\vert\geq 1}\hat{u}_{\alpha,\gamma} r^{\alpha+\gamma}E^{\alpha-\gamma}
\]
and $\hat{n}_{\alpha,\gamma}=\hat{n}_{\alpha,\gamma}(\mu)$ such that
\[
N(\mu,u)=\sum_{\vert \alpha\vert+\vert\gamma\vert\geq 2}\hat{n}_{\alpha,\gamma} r^{\alpha+\gamma}E^{\alpha-\gamma}.
\]
Moreover, $\hat{u}_{\alpha,\gamma}=\hat{u}_{\gamma,\alpha}$ and $\hat{n}_{\alpha,\gamma}=\hat{n}_{\gamma,\alpha}$.
\end{lemma}
\begin{proof}
We set $\hat{u}_{(0,0),(0,0)}=0$ and
\[
    \hat{u}_{(1,0),(0,0)}=\hat{u}_{(0,1),(0,0)}=\hat{u}_{(0,0),(1,0)}=\hat{u}_{(0,0),(0,1)}=\frac{1}{2}.
\]
Then we use the fact that $w=\mathcal{L}N(\mu,v+w)$ and equate orders of $r$ and Fourier coefficients at every order yielding
\begin{equation}\label{eq:uhat}
    \hat{u}_{\alpha,\gamma}=\ell(k_{\alpha,\gamma})\sum_{m=2}^{\vert \alpha\vert+\vert \gamma\vert}\sum_{\substack{\sum_{i=1}^n\alpha^{(i)}=\alpha\\
                                \sum_{i=1}^n\gamma^{(i)}=\gamma}}
    n_{m,\mu}(k_{\alpha^{(1)},\gamma^{(1)}},\ldots,k_{\alpha^{(m)},\gamma^{(m)}})\prod_{i=1}^m\hat{u}_{\alpha^{(i)},\gamma^{(i)}}.
\end{equation}
Similarly we get
\begin{equation}\label{eq:nhat}
     \hat{n}_{\alpha,\gamma}=\sum_{m=2}^{\vert \alpha\vert+\vert \gamma\vert}\sum_{\substack{\sum_{i=1}^n\alpha^{(i)}=\alpha\\
                                \sum_{i=1}^n\gamma^{(i)}=\gamma}}
    n_{m,\mu}(k_{\alpha^{(1)},\gamma^{(1)}},\ldots,k_{\alpha^{(m)},\gamma^{(m)}})\prod_{i=1}^m\hat{u}_{\alpha^{(i)},\gamma^{(i)}}.
\end{equation}
The $\alpha$, $\gamma$ symmetry follows from the fact that $\ell$ and $n_{m,\mu}$ are even.
\end{proof}
\subsection*{Factorization} Using this expansion, we can prove the following factorization of \cref{eq:findim}.
\begin{lemma}\label{lemma:factorization}Let $v\in B^{V}_\varepsilon(0)$ and $\mu\in B^{M}_\varepsilon(\mu_0)$ and define $u=v+w$, where $w\in W^s$ is the solution to \cref{eq:infdim} from \Cref{lemma:LS-reduction}, and Let $2\leq k_1<k_2$ be coprime. Then there exist analytic functions $\Psi_1(r,\theta,\mu)$, $\Psi_2(r,\theta,\mu)$, $\Psi_3(r,\theta,\mu)$ and $\Psi_4(r,\theta,\mu)$ such that
\begin{align}
\langle L(\mu)u+N(\mu,u),\cos(k_1(x+\theta))\rangle&=r_1\Psi_1,\label{eq:c1proj}\\
\langle L(\mu)u+N(\mu,u),\cos(k_2(x+\theta))\rangle&=r_2\Psi_2,\label{eq:c2proj}\\
\langle L(\mu)u+N(\mu,u),\sin(k_1(x+\theta))\rangle&=r_1^{k_2-1}r_2^{k_1}\sin(k_1k_2(\theta_2-\theta_1))\Psi_3,\label{eq:s1proj}\\
\langle L(\mu)u+N(\mu,u),\sin(k_2(x+\theta))\rangle&=r_1^{k_2}r_2^{k_1-1}\sin(k_1k_2(\theta_1-\theta_2))\Psi_4,\label{eq:s2proj}
\end{align}
and
\begin{align*}
\Psi_1(0,\theta,\mu)&=l_\mu(k_1),\\
\Psi_2(0,\theta,\mu)&=l_\mu(k_2),\\
\Psi_3(0,\theta,\mu)&=\hat{n}_{(0,k_1),(k_2-1,0)},\\
\Psi_4(0,\theta,\mu)&=\hat{n}_{(k_2,0),(0,k_1-1)}.
\end{align*}
Moreover, $k_1\Psi_3=k_2\Psi_4$.
\end{lemma}
\begin{proof} We can write
\begin{align*}
\langle L(\mu)u+N(\mu,u),\cos(k_1(x+\theta))\rangle&=\text{Re}\langle L(\mu)u+N(\mu,u),e^{i(k_1(x+\theta))}\rangle\\
&=l_\mu(k_1)r_1+\sum_{\vert \alpha\vert+\vert\gamma\vert\geq 2}r^{\alpha+\gamma}\hat{n}_{\alpha,\gamma}\text{Re}\langle E^{\alpha-\gamma-(1,0)},1\rangle,
\end{align*}
where $\text{Re}\langle E^{\alpha-\gamma-(1,0)},1\rangle=0$ for all $\alpha$ and $\gamma$ such that $\alpha_1+\gamma_1= 0$.
Similarly,
\begin{align*}
\langle L(\mu)u+N(\mu,u),\sin(k_1(x+\theta))\rangle&=\text{Im}\langle L(\mu)u+N(\mu,u),e^{i(k_1(x+\theta))}\rangle\\
&=\sum_{\vert \alpha\vert+\vert\gamma\vert\geq 2}r^{\alpha+\gamma}\hat{n}_{\alpha,\gamma}\text{Im}\langle E^{\alpha-\gamma-(1,0)},1\rangle,
\end{align*}
where $\text{Im}\langle E^{\alpha-\gamma-(1,0)},1\rangle=0$ unless
\[
\alpha_1-\gamma_1-1=-pk_2,\quad \alpha_2-\gamma_2=pk_1\quad\text{and}\quad pk_1k_2(\theta_2-\theta_1)\neq 0.
\]
This means $\alpha_1+\gamma_1\geq k_2-1$ and $\alpha_2+\gamma_2\geq k_1$ with equality precisely for $p=1$, $\alpha=(0,k_1)$ and $\gamma=(k_2-1,0)$. Moreover, when nonzero, we have $\text{Im}\langle E^{\alpha-\gamma-(1,0)},1\rangle=\sin(pk_2k_1(\theta_2-\theta_1))=g_p(\theta)\sin(k_2k_1(\theta_2-\theta_1))$.

The relation between $\Psi_3$ and $\Psi_4$ follows from
\begin{align*}
0&=\langle L(\mu)u+N(\mu,u),v' \rangle\\
&=-r_1k_1\langle L(\mu)u+N(\mu,u),\sin(k_1(x+\theta_1)) \rangle-r_2k_2 \langle L(\mu)u+N(\mu,u),\sin(k_2(x+\theta_2)) \rangle\\
&=r_1^{k_2}r_2^{k_1}\sin(k_1k_2(\theta_1-\theta_2))(k_1\Psi_3-k_2\Psi_4).
\end{align*}
\end{proof}
\begin{remark}
	Note that the assumption that $ k_1$ and $k_2$ are coprime is only required for \cref{eq:c1proj,eq:c2proj}. If $k_1=1$ we can still find $\Psi_3$ and $\Psi_4$ such that \cref{eq:s1proj,eq:s2proj} holds. Thus, the necessary conditions below are still applicable in the latter case, while the sufficient conditions require $ k_1$ and $k_2$ to be coprime.
\end{remark}
\subsection*{Main result} We are now ready to prove the main result of this paper.
\begin{theorem}[Necessary and sufficient conditions]\label{thm:necessaryandsufficient} Consider \cref{eq:fundamentalequation} satisfying \crefrange{property1}{property4}. Then we have the following:
\begin{itemize}[topsep=14pt, itemsep=14pt]
    \item[(a)] If $\vert K_{\mu_0}\vert = 0$, $\vert K_{\mu_0}\vert = 2$, or $K_{\mu_0}=\{\pm k_1,\pm k_2\}$ and $\hat{n}_{(0,k_1),(k_2-1,0)}(\mu_0)\neq 0$, then there exists no asymmetric solutions $u\in X^{s}(\T)$ to \cref{eq:fundamentalequation} for any $\mu\in B_\varepsilon^M(\mu_0)$.
\item[(b)] If $K_{\mu_0}=\{\pm k_1,\pm k_2\}$ for some coprime $k_1$ and $k_2$, $\hat{n}_{(0,k_1),(k_2-1,0)}(\mu_0)=0$, $\mu=(\mu_1,\mu_2,\mu_3,\ldots)$, and
\begin{equation}\label{eq:transversalitycondition}
\left.\frac{\partial(l_\mu(k_1),l_\mu(k_2),\hat{n}_{(0,k_1),(k_2-1,0)}(\mu))}{\partial(\mu_1,\mu_2,\mu_3)}\right\vert_{\mu=\mu_0}=\left. \frac{\partial(\Psi_1,\Psi_2,\Psi_3)}{\partial(\mu_1,\mu_2,\mu_3)}\right\vert_{r=0,\mu=\mu_0}\neq 0,
\end{equation}
then there exists a manifold of solutions $(u,\mu)\in X^{s}(\T)\times M$ to \cref{eq:fundamentalequation} parameterized by $(r,\theta)$. Moreover, $u$ is asymmetric if and only if $(\theta_1-\theta_2)k_1k_2\notin \pi \mathbb{Z}$, the mapping $(r,\theta)\mapsto (u,\mu)$ is analytic and
\[
u=r_1\cos(k_1(x+\theta_1))+r_2\cos(k_2(x+\theta_2))+\mathcal{O}(r^2),\qquad
\mu=\mu_0+\mathcal{O}(r^2).
\]
\end{itemize}
\end{theorem}
\begin{proof} We begin by proving the necessary conditions (a). The first condition is due to the fact that if $\vert K_{\mu_0}\vert = 0$, then there are no nonzero solutions for $\cref{eq:fundamentalequation}$ in a neighborhood of $\mu_0$, by \Cref{lemma:TrivialKernel}. The second condition, if $\vert K_{\mu_0}\vert = 2$, means that every solution $(u,\mu)\in X^s(\T)\times M$ in a neighborhood of $(0,\mu_0)$ is symmetric due to \cref{lemma:KernelSize2}. Now, if $K_{\mu_0}=\{\pm k_1,\pm k_2\}$ then for every asymmetric $v\in B_\varepsilon^V(0)$ and $\mu\in B_\varepsilon^M(\mu_0)$ we construct $u=v+w$ with $w\in W^s$ from \Cref{lemma:LS-reduction} else \cref{eq:infdim} is not satisfied. However, if $\hat{n}_{(0,k_1),(k_2-1,0)}(\mu_0)\neq 0$ then we have $\langle L(\mu)u+N(\mu,u),\sin(k_1(x+\theta))\rangle\neq 0$ so \cref{eq:findim} is not satisfied, whence neither is \cref{eq:fundamentalequation}.

We proceed by proving the sufficient conditions (b). By \cref{lemma:LS-reduction} we obtain a solution $w\in W^s$ to \cref{eq:infdim} for any $v\in B_\varepsilon^V(0)$ and $\mu\in B_\varepsilon^M(\mu_0)$. Then, by our assumptions, we can apply the analytic implicit function theorem to find $\mu\in B_\varepsilon^M(\mu_0)$ such that $(\Psi_1,\Psi_2,\Psi_3)=0$ and thus, by \Cref{lemma:factorization}, \cref{eq:findim} is satisfied for any $v\in B_\varepsilon^V(0)$. Now, parametrizing $V$ with $(r,\theta)$ and using \cref{lemma:asymmetryequivalence} to determine when $u$ is asymmetric gives the desired result.
\end{proof}
\begin{remark}
What was done in \cite{Maehlen_2023} falls somewhere between these two results. The necessary conditions are satisfied, but a weaker condition than \cref{eq:transversalitycondition} is used to prove the existence. However, the numerical results suggest that \cref{eq:transversalitycondition} is true. It should be noted that actually checking \cref{eq:transversalitycondition} is in general very technical. In fact, even finding $\mu_0$ such that $\hat{n}_{(0,k_1),(k_2-1,0)}(\mu_0)=0$ can be very technical.
\end{remark}
\begin{remark} It is possible to imagine a similar result with somewhat different assumptions on the equation, in particular, \cref{property3,property4}. The main use of \Cref{property3} is to show that $\Psi_4=0$ if and only if $\Psi_3=0$. If this equivalence does not hold, it would be possible to solve $(\Psi_1,\Psi_2,\Psi_3,\Psi_4)=0$ instead, analogously to how we solve $(\Psi_1,\Psi_2,\Psi_3)=0$ in the theorem. At least if the equation contains another parameter with the correct properties. This would give another, but similar, result for sufficient conditions without assuming \cref{property3}. 

Moreover, we can also imagine a simple modification that allows us to lower the regularity in \cref{property4}. If we know that we are interested in some particular wavenumbers $k_1$ and $k_2$ a priori, then we can work with $F\in C^{k_1+k_2}(M\times X^s; X^t)$. The expansions in \cref{lemma:expansions} simply have to be truncated.
If we also restrict ourselves to $r$ so that
\[
r_1\lesssim r_2 \lesssim r_1.
\]
then the remainder immediately gets the necessary asymptotic behavior as $r\to 0$. It may be possible without this restriction, but then some estimates for the remainder are necessary to obtain the factorizations in \cref{lemma:factorization}. In the end, we also only end up with solutions $(u,\mu)$, which are $C^1$ with respect to $(r,\theta)$.
\end{remark}
\section{Application to the infinite depth water wave problem}\label{sec:example}
\subsection*{An infinite depth Whitham equation} Before considering the water wave problem, we begin by studying a modified version of the capillary-gravity Whitham equation obtained by introducing a depth parameter, $d$, which we let tend to infinity. This allows us to contrast the result in \cite{Maehlen_2023} with an infinite depth result. The purpose of this is to observe the same relation between finite depth and infinite depth for the water wave problem.

This depth dependent capillary-gravity Whitham equation is given by
\begin{equation}
u_t+(M_{T,d}u+u^2)_x=0.
\end{equation}
As in \cref{eq:Whitham}, $M_{T,d}$ is a spatial Fourier multiplier
\[
\widehat{M_{T,d} u}(t, \xi) = m_{T,d} (\xi)\widehat{u}(t, \xi),
\]
but here the symbol $m_{T,d}$ is given by
\[
m_{T,d}(\xi)=\sqrt{\frac{(1+T\xi^2)\tanh(\xi d)}{\xi}}
\]

    We assume periodic solutions with period $2\pi/\kappa$ and rescale the period to $2\pi$ by changing variables. Moreover, we introduce the traveling wave assumption and integrate the equation that yields
    \begin{equation}\label{eq:WhithamVarDepth}
        -cu+M_{\kappa,T,d}u+u^2=0,
    \end{equation}
    where
    \[
        M_{\kappa,T,d}e^{ikx}=m_{T,d}(\kappa k)e^{ikx}=\sqrt{\frac{(1+T\kappa^2k^2)\tanh(\kappa k d)}{\kappa k}}e^{ikx}.
    \]
For this equation we use the Zygmund spaces, so $X^s(\T)=\mathcal{C}^s(\T)$ for some $s>1$.
We begin by noting that for an arbitrary but finite depth this equation can be transformed back into the original equation by the change of variables
\[
   u\mapsto\tilde{u}=u/d^{1/2},\qquad c\mapsto\tilde{c}=c/d^{1/2},\qquad \kappa\mapsto\tilde{\kappa}=\kappa d,\qquad\text{and}\qquad T\mapsto\tilde{T}=T/d^2.
\]
This turns \cref{eq:WhithamVarDepth} into
\begin{equation}\label{eq:WhithamDepth1}
        d(-\tilde{c}\tilde{u}+M_{\tilde{\kappa},\tilde{T},1}\tilde{u}+\tilde{u}^2)=0,
\end{equation}
which is equivalent to the equation considered in \cite{Maehlen_2023}. Whence for a general finite depth we get a solution if and only if we have a solution to the equation with normalized depth. Assuming that we have a sequence of asymmetric solutions with depth tending to infinity but with bounded surface tension, we arrive at a contradiction; through the change of variables above, we would obtain a sequence of solutions with normalized depth and surface tension tending to zero. This is not expected, as in \cite{Maehlen_2023} it is shown that $\lim_{T\to 0}\hat{n}_{(0,k_1),(k_2-1,0)}=\pm\infty$.

The argument above shows that we cannot reach an asymmetric infinite depth solution to \cref{eq:WhithamVarDepth} as the limit of finite depth solutions. However, asymmetric infinite depth solutions could exist without being the limit of finite depth solutions. To rule out this possibility, we also return to \cref{eq:WhithamVarDepth}, but now the plan is to take the limit $d\to \infty$ before solving the equation. However, to do so, we change the equation to one in the Zygmund space of functions of zero mean
\[
\ac{\circ}{\mathcal{C}}^s(\T)=\left\{u\in\mathcal{C}^s(\T):\frac{1}{\vert \T\vert}\int_\T u\,dx=0\right\}
\]
obtaining
\[
-cu+M_{\kappa,T,d}u+\ac{\circ}{P}u^2=0,
\]
where $u\in\ac{\circ}{\mathcal{C}}^s(\T)$ and $\ac{\circ}{P}$ denotes the projection from ${\mathcal{C}}^s(\T)$ onto $\ac{\circ}{\mathcal{C}}^s(\T)$. It is possible to show that this equation is equivalent to \cref{eq:WhithamVarDepth} by a Galilean transformation and redefining the constant of integration. The reason for rewriting the equation this way is that $m_{T,d}(0)$ is defined through a limit, which does not commute with the limit $d\to \infty$. Thus, with zero mean, we can, without ambiguity, take the limit $d\to \infty$ and obtain
\begin{equation}\label{eq:WhithamInfDepth}
-cu+M_{\kappa,T,\infty}u+\ac{\circ}{P}u^2=0
\end{equation}
where the Fourier multiplier $M_{\kappa,T,\infty}$ now is given by
\[
M_{\kappa,T,\infty}e^{ikx}=m_{T,\infty}(\kappa k)e^{ikx}= \sqrt{\frac{1}{\kappa\vert k\vert}+T\kappa\vert k\vert}e^{ikx}.
\]
We begin by checking that this equation belongs to the class of equations considered in this paper.
\begin{lemma} The \crefrange{property1}{property4} are satisfied for \cref{eq:WhithamInfDepth} with $X^{s}(\T)=\ac{\circ}{\mathcal{C}}^s(\T)$, $X^{t}(\T)=\ac{\circ}{\mathcal{C}}^{s-1/2}(\T)$, $s>1$, where $\mu=(c,\kappa,T)$, $L(\mu)u=-cu+M_{\kappa,T,\infty}u$, and $N(\mu,u)=\ac{\circ}{P}u^2$.
\end{lemma}
\begin{proof}
For \cref{property1} we begin by noting that the linear part is obviously a Fourier multiplier with even symbol. Moreover, we have
\[
\vert k\vert^{1/2}\lesssim m_{T,\infty}(\kappa k)\lesssim \vert k\vert^{1/2},
\]
which means that $M_{\kappa, T,\infty}:\ac{\circ}{\mathcal{C}}^{s}(\T)\to\ac{\circ}{\mathcal{C}}^{s-1/2}(\T)$ is invertible. It follows that the mapping $u\mapsto u-cM^{-1}_{\kappa,T,\infty}u$ as an operator on $\ac{\circ}{\mathcal{C}}^{s}(\T)$ is the identity mapping plus a compact map due to the compact embedding $\ac{\circ}{\mathcal{C}}^{s+1/2}(\T)\subset\subset\ac{\circ}{\mathcal{C}}^{s}(\T)$, so the map is Fredholm and so is
\[
L(\mu)=M_{\kappa,T,\infty}\circ(I-cM^{-1}_{\kappa,T,\infty}).
\]

For \cref{property2} we note that
\[
\ac{\circ}{P}u^2=N_2(u,u),
\]
where $N_2(e^{ik_1x},e^{ik_2x})=n_2(k_1,k_2)e^{i(k_1+k_2)x}$ for
\[
n_2(k_1,k_2)=\left\{
        \begin{aligned}
        &0 &&\text{if }k_1+k_2=0,\\
        &1 &&\text{if }k_2+k_2\neq 0.
        \end{aligned}
        \right.
\]

    \Cref{property3} follows because we can define the functional $\mathcal{J}_\mu:\ac{\circ}{\mathcal{C}}^{s}\to \R$ by
\[
\mathcal{J}_\mu(u)=\int_\T \frac{1}{2}u(M_{\kappa,T,\infty}-c)u+\frac{1}{3}u^3\, dx.
\]

    For \cref{property4} we use the fact that the map $u\mapsto F(\mu,u)$ is quadratic and only the linear part depends on $\mu$. Moreover, $(\kappa,T)\to m_{T,\infty}(\kappa k)$ is analytic around any point $(\kappa_0,T_0)\in \R_+^2$, with a radius of convergence that is uniformly bounded below for all $\vert k\vert\geq 1$, so $(c,\kappa,T)\mapsto -c+M_{\kappa,T,\infty}\in \mathcal{L}(\ac{\circ}{\mathcal{C}}^s(\T),\ac{\circ}{\mathcal{C}}^{s-1/2}(\T))$ is analytic. That is, we can write
    \[
        -cu+M_{\kappa,T,\infty}u+\ac{\circ}{P}u^2=-c_0u-(c-c_0)u+\ac{\circ}{P}u^2+\sum_{p,q\geq 0}(\kappa-\kappa_0)^p(T-T_0)^q M^{p,q}u,
    \]
    where $M^{p,q}$ is a Fourier multiplier acting on $\ac{\circ}{\mathcal{C}}^s$ through $e^{ikx}\mapsto  \frac{\partial_{\kappa}^p\partial_{T}^qm_{T_0,\infty}(\kappa_0 k)}{p!q!}e^{ikx}$.
\end{proof}
With this result in hand, we can prove the following result.
\begin{proposition}\label{prop: c0AndKappa0AreDeterminedByK1K2AndT}
    For any $T\in \R_+$ and integers $1\leq k_1<k_2$ there exists a unique pair $(c_0,\kappa_0)\in \mathbb{R}_+^2$ such that
    \begin{align*}
        m_{T,\infty}(\kappa_0 k_1) = c_0 = m_{T,\infty}(\kappa_0 k_2).
    \end{align*}
We further have $m_{T,\infty}'(\kappa_0 k_1)<0<m_{T,\infty}'(\kappa_0 k_2)$, and $c_0=c_0(T ;k_1,k_2)$ and $\kappa_0=\kappa_0(T;k_1,k_2)$ are explicitly given by
    \begin{align*}
        c_0 = T^{1/4}\sqrt{\sqrt{\frac{k_1}{k_2}}+\sqrt{\frac{k_2}{k_1}}} \quad \text{ and }\quad \kappa_0 =\frac{1}{\sqrt{k_1k_2 T}}.
    \end{align*}
\end{proposition}
\begin{proof} We begin by solving $m_{T,\infty}(\kappa k_1)=m_{T,\infty}(\kappa k_2)$, which is equivalent to
\[
\frac{1}{k_1\kappa}+Tk_1\kappa=\frac{1}{k_2\kappa}+Tk_2\kappa.
\]
This equation has a unique positive solution
\[
\kappa_0 =\frac{1}{\sqrt{k_1k_2 T}}
\]
Evaluating $m_{T,\infty}(\kappa k_1)$ or $m_{T,\infty}(\kappa k_2)$ gives
\[
c_0 = T^{1/4}\sqrt{\sqrt{\frac{k_1}{k_2}}+\sqrt{\frac{k_2}{k_1}}}.
\]
for $\xi>0$ we have
\[
m_{T,\infty}'(\xi)=\frac{-\frac{1}{\xi^2}+T}{2m_{T,\infty}(\xi)}.
\]
Evaluating at $\kappa_0 k_1$ and $\kappa_0 k_2$, we obtain $m_{T,\infty}'(\kappa_0 k_1)<0<m_{T,\infty}'(\kappa_0 k_2)$.
\end{proof}
Here we have explicit expressions for $c_0$ and $\kappa_0$, which means that we can prove stronger results about $\hat{n}_{(0,k_1),(k_2-1,0)}$ than done for $d=1$ in \cite{Maehlen_2023}.
\begin{proposition}\label{prop:Tdependence} For any $(c_0, \kappa_0, T_0)$ such that $\dim \ker{(-c_0+M_{\kappa_0,T_0,\infty})}=4$ there exists a constant $C=C(k_1,k_2)$ such that
\[
\hat{n}_{(0,k_1),(k_2-1,0)}(c_0, \kappa_0, T_0)=CT_0^{-\frac{k_1+k_2-3}{4}}.
\]
\end{proposition}
\begin{proof} From \Cref{prop: c0AndKappa0AreDeterminedByK1K2AndT} we have $c_0=T_0^{1/4}\sqrt{\sqrt{\frac{k_1}{k_2}}+\sqrt{\frac{k_2}{k_1}}}$ and $\kappa_0=\frac{1}{\sqrt{k_1k_2T_0}}$. It follows that
\[
\ell(k_{\alpha,\gamma})=\frac{1}{-c_0+m_{T,\infty}(\kappa_0 k_{\alpha,\gamma})}=\frac{1}{T_0^{1/4}\left(-\sqrt{\sqrt{\frac{k_1}{k_2}}+\sqrt{\frac{k_2}{k_1}}}+\sqrt{\frac{\sqrt{k_1k_2}}{\vert k_{\alpha,\gamma}\vert}+\frac{\vert\kappa_{\alpha,\gamma}\vert}{\sqrt{k_1k_2}}}\right)}.
\]
This allows us to show that there exist constants $C_{\alpha,\gamma}=C_{\alpha,\gamma}(k_1,k_2)$ such that
\[
\hat{u}_{\alpha,\gamma}=C_{\alpha,\gamma}T_0^{-\frac{\vert \alpha\vert +\vert \gamma \vert -1}{4}}.
\]
We do this by induction over $\vert \alpha \vert +\vert \gamma\vert$. Clearly it is true for $\vert \alpha \vert +\vert \gamma\vert=1$ and assuming that it is true for all $\vert \alpha\vert+\vert \gamma\vert< m$ we obtain, by \cref{eq:uhat}, for any $\vert \alpha \vert +\vert \gamma\vert=m$ such that $\alpha\neq \gamma$
\begin{align*}
\hat{u}_{\alpha,\gamma}&=\ell(k_{\alpha,\gamma})\sum_{\substack{\alpha'+\alpha''=\alpha\\
                                    \gamma'+\gamma''=\gamma}}C_{\alpha',\gamma'}C_{\alpha'',\gamma''}T_0^{-\frac{\vert \alpha'\vert+\vert \alpha''\vert +\vert \gamma' \vert+\vert \gamma''\vert -2}{4}}\\&=T_0^{-\frac{\vert \alpha\vert +\vert \gamma \vert -1}{4}}\frac{\sum_{\substack{\alpha'+\alpha''=\alpha\\
                                    \gamma'+\gamma''=\gamma}}C_{\alpha',\gamma'}C_{\alpha'',\gamma''}}{\left(-\sqrt{\sqrt{\frac{k_1}{k_2}}+\sqrt{\frac{k_2}{k_1}}}+\sqrt{\frac{\sqrt{k_1k_2}}{\vert k_{\alpha,\gamma}\vert}+\frac{\vert\kappa_{\alpha,\gamma}\vert}{\sqrt{k_1k_2}}}\right)},
\end{align*}
and if $\alpha=\gamma$ we obtain $C_{\alpha,\gamma}=0$. Similarly, with $T_0$ dependence for $\hat{u}_{\alpha,\gamma}$ proved, \cref{eq:nhat} gives us
\[
\hat{n}_{(0,k_1),(k_2-1,0)}=T^{-\frac{k_1+k_2-3}{4}}\sum_{\substack{\alpha'+\alpha''=(0,k_1)\\
                                    \gamma'+\gamma''=(k_2-1,0)}}C_{\alpha',\gamma'}C_{\alpha'',\gamma''}
\]
proving the result.
\end{proof}
With \cref{prop: c0AndKappa0AreDeterminedByK1K2AndT,prop:Tdependence} we can apply \cref{thm:necessaryandsufficient} and find that there are no arbitrarily small asymmetric solutions to \cref{eq:WhithamInfDepth}.
\begin{theorem} Let $C(k_1,k_2)$ be the constant from \cref{prop:Tdependence}.
   If $C(k_1,k_2)\neq 0$ for all integers $1\leq k_1<k_2$, then there exist no arbitrarily small asymmetric solutions to \cref{eq:WhithamInfDepth} in $\ac{\circ}{\mathcal{C}}^s$, $s>1$.    
\end{theorem}
\begin{proof} Note that $m'_{T,\infty}(\xi)=0$ for exactly one $\xi\in \R_+$ so it is clear that $\dim \ker{(-c_0+M_{\kappa_0,T_0,\infty})}\leq 4$. Thus, for any wavenumber pair $(k_1,k_2)$ such that the constant from \Cref{prop:Tdependence} is nonzero, we can apply \cref{thm:necessaryandsufficient} to show that there are no asymmetric solutions for that wavenumber pair. If this is true for all wavenumber pairs, then there are no asymmetric solutions.
\end{proof}
\begin{remark} We have calculated $C$ numerically for all $(k_1,k_2)$ such that $k_1<k_2\leq 100$ and found no examples where $C=0$ among these, and we also believe it is highly unlikely that $C=0$ for any of the remaining $(k_1,k_2)$. Moreover, the condition $C=0$ is not sufficient to show the existence of small-amplitude asymmetric waves, rather, even if $C$ happens to be zero for some wavenumber pair, it is still very unlikely that asymmetric solutions exist. 
\end{remark}

Another way to view this is that one independent parameter is fixed by taking the limit $d\to \infty$. In \cref{eq:WhithamVarDepth} we seemingly have four parameters $(c,\kappa,T,d)$, but because we can transform the equation into \cref{eq:WhithamDepth1} we only have three independent parameters. However, taking the limit fixes one of these parameters. By changing the variables
\[
u\mapsto\tilde{u}=\kappa^{1/2} u,\qquad c\mapsto\tilde{c}=\kappa^{1/2} c,\qquad \text{and}\qquad T\mapsto\tilde{T}=\kappa^2 T,
\]
we turn \cref{eq:WhithamInfDepth} into
\[
\frac{1}{\kappa}(-\tilde{c}\tilde{u}+M_{1,\tilde{T},\infty}\tilde{u}+\ac{\circ}{P}u^2)=0,
\]
where it is obvious that $\tilde{c}$ and $\tilde{T}$ are the only free parameters left. Using this equation instead, we would have to fix $\tilde{c}_0$ and $\tilde{T}_0$ to solve
\[
m_{\tilde{T}_0,\infty}(k_1)=\tilde{c}_0=m_{\tilde{T}_0,\infty}(k_2)
\]
and thus $\hat{n}_{(0,k_1),(k_2-1,0)}$ becomes a constant which plays a role equivalent to $C$ from \cref{prop:Tdependence}.
\subsection*{The Babenko equation} 
The full water wave problem can be expressed as a single equation known as the Babenko equation \cite{K.I.Babenko1987}. Since its inception as an equation only for gravity waves it has evolved and been derived in other regimes. Here we will consider the finite depth version with surface tension \cite{Buffoni_2004,Clamond_2015} and the infinite depth version with surface tension \cite{Buffoni_2000}. We show that they share a relation similar to the infinite and finite depth Whitham equations considered above and in \cite{Maehlen_2023}, respectively. This will allow us to conclude that there are no arbitrarily small asymmetric solutions to the water wave problem on infinite depth, and give hope to the possibility of finding arbitrarily small asymmetric solutions to the finite depth water wave problem.

We begin by introducing the finite depth Babenko equation with surface tension \cite{Buffoni_2004,Clamond_2015}, normalized for zero mean and rescaled to the Zygmund spaces on the unit circle in the same way as the Whitham equation
\begin{equation}\label{eq:BabFin}
\begin{aligned}
0&=-c^2\mathcal{H}_{\kappa,d}\tilde{u}+g\{\tilde{u}+\ac{\circ}{P}\tilde{u}\mathcal{H}_{\kappa,d}\tilde{u}+\ac{\circ}{P}\mathcal{H}_{\kappa,d}(\tilde{u}^2/2)\}\\
&\qquad-\ac{\circ}{P}T\mathcal{D}_{\kappa}\left\{\frac{\mathcal{D}_{\kappa}\tilde{u}}{\sqrt{(\mathcal{D}_{\kappa}\tilde{u})^2+(1+\mathcal{H}_{\kappa,d}\tilde{u})^2}}\right\}+ \ac{\circ}{P}T \mathcal{H}_{\kappa,d}\left\{\frac{1+\mathcal{H}_{\kappa,d}\tilde{u}}{\sqrt{(\mathcal{D}_{\kappa}\tilde{u})^2+(1+\mathcal{H}_{\kappa,d}\tilde{u})^2}}\right\},
\end{aligned}
\end{equation}
where, $c$ is the wave speed, $g$ is the gravitational acceleration, $T$ is the surface tension strength.
Moreover, the operators in the equation are given by
\begin{align*}
\mathcal{H}_{\kappa,d}(e^{ikx})&=\kappa k \coth(\kappa k d)e^{ikx},\\
\mathcal{D}_{\kappa}(e^{ikx})&=i\kappa k e^{ikx},
\end{align*}
where $d$ is the depth. We take the limit $d\to \infty$ and introduce the new variables
\[
\tilde{u}\mapsto u=\kappa \tilde{u}, \qquad c\mapsto \nu=\frac{c \kappa^{1/2}}{g^{1/2}}, \qquad\text{and}\qquad T\mapsto \beta=\frac{T\kappa^2}{g}.
\]
This gives us the equation
\begin{equation}\label{eq:BabInf}
\begin{aligned}
0&=\frac{g}{\kappa}\Bigg(-\nu^2\mathcal{H}_{1,\infty}u+u+\ac{\circ}{P}u\mathcal{H}_{1,\infty}u+\ac{\circ}{P}\mathcal{H}_{1,\infty}(u^2/2)\\
&\qquad-\ac{\circ}{P}\beta \mathcal{D}_{1}\left\{\frac{\mathcal{D}_{1}u}{\sqrt{(\mathcal{D}_{1}u)^2+(1+\mathcal{H}_{1,\infty}u)^2}}\right\}+ \ac{\circ}{P}\beta \mathcal{H}_{1,\infty}\left\{\frac{1+\mathcal{H}_{1,\infty}u}{\sqrt{(\mathcal{D}_{1}u)^2+(1+\mathcal{H}_{1,\infty}u)^2}}\right\}\Bigg),
\end{aligned}
\end{equation}
where
\[
\mathcal{H}_{1,\infty}(e^{ikx})= \vert k\vert e^{ikx}.
\]
\Cref{eq:BabInf} is of the form of \cref{eq:fundamentalequation}
with
\begin{align}
\mu&=(\nu,\beta)\in\R_+^2,\label{eq:InfWW1}\\
L(\mu)u&=-\nu^2\mathcal{H}_{1,\infty}u+u-\beta \mathcal{D}_{1}^2 u,\label{eq:InfWW2}\\
\nonumber N(\mu,u)&=\ac{\circ}{P}\Bigg(u\mathcal{H}_{1,\infty}u+\mathcal{H}_{1,\infty}(u^2/2)+\beta\mathcal{D}_1^2 u\\
&\qquad-\beta \mathcal{D}_{1}\left\{\frac{\mathcal{D}_{1}u}{\sqrt{(\mathcal{D}_{1}u)^2+(1+\mathcal{H}_{1,\infty}u)^2}}\right\}+\beta \mathcal{H}_{1,\infty}\left\{\frac{1+\mathcal{H}_{1,\infty}u}{\sqrt{(\mathcal{D}_{1}u)^2+(1+\mathcal{H}_{1,\infty}u)^2}}\right\}\Bigg)\label{eq:InfWW3}
\end{align}
We begin by showing that this equation belongs to the class of equations treated in this paper.
\begin{lemma}The \crefrange{property1}{property4} are satisfied for \cref{eq:BabInf} where the parameters, linear part and nonlinear part are given by \crefrange{eq:InfWW1}{eq:InfWW3} with $X^{s}(\T)=\ac{\circ}{\mathcal{C}}^s(\T)$, $X^{t}(\T)=\ac{\circ}{\mathcal{C}}^{s-2}(\T)$, $s>5/2$.
\end{lemma}
\begin{proof} To show \cref{property1} is very similar to how it was done for the Whitham equation. Immediately we get that $L(\mu)$ is a Fourier multiplier with symbol
\[
l_\mu(k)=-\nu^2 \vert k\vert +1+\beta k^2.
\]
Moreover, $I-\beta \mathcal{D}_1^2$ is invertible as an operator from $\ac{\circ}{\mathcal{C}}^s$ to $\ac{\circ}{\mathcal{C}}^{s-2}$ and $\mathcal{H}_{1,\infty}$ maps $\ac{\circ}{\mathcal{C}}^s$ to $\ac{\circ}{\mathcal{C}}^{s-1}$. This means the operator $-\nu^2(I-\beta \mathcal{D}_1^2)^{-1}\mathcal{H}_{1,\infty}$ is compact on $\ac{\circ}{\mathcal{C}}^s$, whence $I-\nu^2(I-\beta \mathcal{D}_1^2)^{-1}\mathcal{H}_{1,\infty}$ and thus so is 
\[
L(\mu)=(I-\beta \mathcal{D}_1^2)(I-\nu^2(I-\beta \mathcal{D}_1^2)^{-1}\mathcal{H}_{1,\infty}).
\]

Verification of \cref{property2} is certainly a bit more involved in this case than in the Whitham case. However, if $u\in B^{\ac{\circ}{\mathcal{C}}^s}_\varepsilon(0)$ we have
	\[
	\frac{1}{\sqrt{1+(2\mathcal{H}_{1,\infty} u+(\mathcal{H}_{1,\infty} u)^2+(\mathcal{D}_{1} u)^2)}}=1+\sum_{m=1}^{\infty}a_m (2\mathcal{H}_{1,\infty} u+(\mathcal{H}_{1,\infty} u)^2+(\mathcal{D}_{1} u)^2)^m
	\]
	for some $a_m\in \R$. Expanding the powers and using $a_1=-\frac{1}{2}$ we can rearrange this into
	\begin{equation}\label{eq:babenkoExpansion}
		\frac{1}{\sqrt{1+(2\mathcal{H}_{1,\infty} u+(\mathcal{H}_{1,\infty} u)^2+(\mathcal{D}_{1} u)^2)}}=1-\mathcal{H}_{1,\infty} u+\sum_{j,k:j+2k\geq2}b_{j,k} (\mathcal{H}_{1,\infty} u)^j(\mathcal{D}_{1} u)^{2k},
	\end{equation}
	for some $b_{j,k}\in \R$, which allows us to find
	\begin{align*}
		N_2(\mu,(u_1,u_2))&= u_1\mathcal{H}_{1,\infty} u_2+ \mathcal{H}_{1,\infty}(u_1u_2/2)+\beta \mathcal{D}_{1} (\mathcal{D}_{1} u_1\mathcal{H}_{1,\infty} u_2)-\beta \mathcal{H}_{1,\infty}(\mathcal{D}_{1} u_1\mathcal{D}_{1} u_2/2)\\
	\intertext{ and }
		N_m(\mu,(u_1,\ldots,u_n))&=\sum_{j,k:j+2k+1=m}\beta \left[(b_{j+1,k}+b_{j,k})\mathcal{H}_{1,\infty}(\mathcal{H}_{1,\infty} u_1\ldots \mathcal{H}_{1,\infty} u_{j+1}\mathcal{D}_{1} u_{j+2}\ldots \mathcal{D}_{1} u_m)\right.\\
		&\qquad\qquad\qquad\qquad\left.-b_{j,k}\mathcal{D}_{1} (\mathcal{H}_{1,\infty} u_1\ldots \mathcal{H}_{1,\infty} u_{j}\mathcal{D}_{1} u_{j+1}\ldots \mathcal{D}_{1} u_m)\right]
	\end{align*}
	for $m\geq 3$. This means that
 \begin{align*}
		N_2(\mu,(e^{ik_1x},e^{ik_2x}))&= (\vert k_2\vert+\vert k_1+k_2\vert/2-(k_1+k_2)k_1\vert k_2\vert+\beta\vert k_1+k_2\vert k_1k_2/2)e^{i(k_1+k_2)x}\\
	\intertext{ and }
		N_m(\mu,(e^{ik_1x},\ldots,e^{ik_nx}))&=\sum_{j,k:j+2k+1=m}\beta (-1)^k\left[(b_{j+1,k}+b_{j,k})\left\vert \sum_{l=1}^m k_l \right\vert\left(\prod_{l=1}^{j+1}\vert k_l\vert\right)\left( \prod_{l=j+2}^{m} k_{l}\right)\right.\\
		&\qquad\qquad\qquad\qquad\left.+b_{j,k}\left(\sum_{l=1}^m k_l \right)\left(\prod_{l=1}^{j}\vert k_l\vert \right)\left(\prod_{l=j+1}^{m} k_{l}\right)\right]e^{i(\sum_{l=1}^m k_l)x},
	\end{align*}
for $n\geq 3$. These expressions define $n_{m,\mu}(k_1,\ldots,k_n)$ (unless $\sum_{l=1}^m k_l=0$, then $n_{m,\mu}(k_1,\ldots,k_n)=0$) with the desired symmetry properties.

For \cref{property3} we have the functional \cite{Buffoni_2000} on $\ac{\circ}{\mathcal{C}}^s(\T)$
\[
\mathcal{J}_\mu(u)=\int_{\T}\frac{1}{2}\left[u^2(1+\mathcal{H}_{1,\infty}u)-\nu^2 u\mathcal{H}_{1,\infty}u\right]+\beta\sqrt{(\mathcal{D}_1u)^2 +(1+\mathcal{H}_{1,\infty}u)^2} \, dx.
\]

For \cref{property4}, we can use the expansion in \cref{eq:babenkoExpansion} to see that we have a series expansion for sufficiently small $u$. Then the joint analyticity in $u$ and the parameters follows from the fact that the equation is simply linear in $\beta$ and quadratic in $\nu$.
\end{proof}
The linear part of this equation is qualitatively the same as for the Whitham equation, but for completeness we include the following result. 
\begin{proposition}\label{prop:nu0andbeta0fork1k2}
    For any integers $1\leq k_1<k_2$ there exists a unique pair $(\nu_0,\beta_0)\in \mathbb{R}_+^2$ such that
    \begin{align*}
        l_{\mu_0}(k_1)=l_{\mu_0}(k_2)=0.
    \end{align*}
We further have $l_{\mu_0}(k)\neq 0$ for all $k\notin\{\pm k_1,\pm k_2\}$ and $\nu_0=\nu_0(k_1,k_2)$ and $\beta_0=\beta_0(k_1,k_2)$ are explicitly given by
    \begin{align*}
        \nu_0 = \sqrt{\frac{1}{k_2}+\frac{1}{k_1}} \quad \text{ and }\quad \beta_0 =\frac{1}{k_1k_2}.
    \end{align*}
\end{proposition}
\begin{proof}
    The result comes from solving the system of equations
\begin{align*}
    -\nu^2k_1+1+\beta k_1^2=0,\\
    -\nu^2k_2+1+\beta k_2^2=0,
\end{align*}
except the part that $l_{\mu_0}(k)\neq 0$ for all $k\notin\{\pm k_1,\pm k_2\}$. However, if $l_{\mu_0}(k)=0$ for such a $k$ we would have a quadratic equation with three roots.
\end{proof}
This allows us to prove the following result about the infinite depth Babenko equation.
\begin{theorem}
   If the constants $\hat{n}_{(0,k_1),(k_2-1,0)}(\mu_0(k_1,k_2))\neq 0$ for all integers $1\leq k_1<k_2$, then there exists no arbitrarily small asymmetric solutions to \cref{eq:BabInf} in $\ac{\circ}{\mathcal{C}}^s(\T)$, $s>5/2$.
\end{theorem}
\begin{proof} By \Cref{prop:nu0andbeta0fork1k2} we have $\dim \ker (L(\mu))\leq 4$. Thus, we only have to deal with the $\dim \ker (L(\mu))= 4$ case to show the nonexistence of solutions. For any $(k_1,k_2)$ we can apply \cref{thm:necessaryandsufficient} to find that there are no small-amplitude solutions to \cref{eq:BabInf} in $\ac{\circ}{\mathcal{C}}^s(\T)$ for that specific pair $(k_1,k_2)$ if $\hat{n}_{(0,k_1),(k_2-1,0)}(\mu_0(k_1,k_2))\neq 0$. If $\hat{n}_{(0,k_1),(k_2-1,0)}(\mu_0(k_1,k_2))\neq 0$ for all $(k_1,k_2)$, then there are no small-amplitude solutions at all to \cref{eq:BabInf} in $\ac{\circ}{\mathcal{C}}^s(\T)$.
\end{proof}
\begin{remark}
    Just as in the case for the Whitham equation we expect $\hat{n}_{(0,k_1),(k_2-1,0)}(\mu_0(k_1,k_2))\neq 0$ to hold for all $(k_1,k_2)$. Moreover, $\hat{n}_{(0,k_1),(k_2-1,0)}(\mu_0(k_1,k_2))= 0$ for some $(k_1,k_2)$ does not imply the existence of arbitrarily small asymmetric solutions.
\end{remark}
By \cite{Buffoni_2000} we know that \cref{eq:BabInf} is equivalent to the infinite depth water wave problem, so we obtain the result that the infinite depth water wave problem does not have arbitrarily small asymmetric solutions. This result in itself is not new; see \cite[Theorem 4.5]{Okamoto_2001}. However, in the context it is presented here we see the parallel with the finite and infinite depth Whitham equation. That we reduce the number of available parameters by taking the limit $d\to \infty$. Thus, in the finite depth case, \cref{eq:BabFin}, there is effectively one more parameter in the formulation of the water wave problem than in the model considered in \cite{Okamoto_2001}. This could allow for a similar existence result to the one in \cite{Maehlen_2023} for the water wave problem.
\section*{Acknowledgemets} The author is supported by the Research Council of Norway under Grant Agreement No. 325114. The author would also like to thanks Mats Ehrnström for insightful discussions.
\printbibliography

@Article{Maehlen_2023,
  author    = {Mæhlen, Ola and Seth, Douglas Svensson},
  title     = {Asymmetric travelling wave solutions of the capillary-gravity Whitham Equation},
  year      = {2023},
  copyright = {arXiv.org perpetual, non-exclusive license},
  doi       = {10.48550/ARXIV.2312.15343},
  publisher = {arXiv},
}

@Article{Buffoni_2004,
  author    = {Buffoni, B.},
  journal   = {Archive for Rational Mechanics and Analysis},
  title     = {Existence and Conditional Energetic Stability of Capillary-Gravity Solitary Water Waves by Minimisation},
  year      = {2004},
  issn      = {1432-0673},
  month     = mar,
  number    = {1},
  pages     = {25--68},
  volume    = {173},
  doi       = {10.1007/s00205-004-0310-0},
  publisher = {Springer Science and Business Media LLC},
}

@Article{Clamond_2015,
  author    = {Clamond, Didier and Dutykh, Denys and Durán, Angel},
  journal   = {Journal of Fluid Mechanics},
  title     = {A plethora of generalised solitary gravity–capillary water waves},
  year      = {2015},
  issn      = {1469-7645},
  month     = nov,
  pages     = {664--680},
  volume    = {784},
  doi       = {doi:10.1017/jfm.2015.616},
  publisher = {Cambridge University Press (CUP)},
}

@Article{Buffoni_2000,
  author    = {Buffoni, B. and Dancer, E. N. and Toland, J. F.},
  journal   = {Archive for Rational Mechanics and Analysis},
  title     = {The Regularity and Local Bifurcation of Steady Periodic Water Waves},
  year      = {2000},
  issn      = {1432-0673},
  month     = jun,
  number    = {3},
  pages     = {207--240},
  volume    = {152},
  doi       = {https://doi.org/10.1007/s002050000086},
  publisher = {Springer Science and Business Media LLC},
}

@Article{K.I.Babenko1987,
  author  = {K.~I.~Babenko},
  journal = {Dokl. Akad. Nauk SSSR},
  title   = {Some remarks on the theory of surface waves of finite amplitude},
  year    = {1987},
  number  = {1033--1037},
  volume  = {294},
}

@Book{Buffoni_2016,
  author    = {Buffoni, Boris and Toland, John},
  publisher = {Princeton University Press},
  title     = {Analytic Theory of Global Bifurcation},
  year      = {2016},
  isbn      = {1400884330},
  month     = sep,
  doi       = {https://doi.org/10.2307/j.ctt1dwstqb},
}

@Article{Gao_2016Investigations,
  author    = {Gao, T. and Wang, Z. and Vanden-Broeck, J.-M.},
  journal   = {Journal of Fluid Mechanics},
  title     = {Investigation of symmetry breaking in periodic gravity–capillary waves},
  year      = {2016},
  issn      = {1469-7645},
  month     = dec,
  pages     = {622--641},
  volume    = {811},
  doi       = {10.1017/jfm.2016.751},
  publisher = {Cambridge University Press (CUP)},
}

@Article{Gao_2016Solitary,
  author    = {Gao, T. and Wang, Z. and Vanden-Broeck, J.-M.},
  journal   = {Proceedings of the Royal Society A: Mathematical, Physical and Engineering Sciences},
  title     = {On asymmetric generalized solitary gravity–capillary waves in finite depth},
  year      = {2016},
  issn      = {1471-2946},
  month     = oct,
  number    = {2194},
  pages     = {20160454},
  volume    = {472},
  doi       = {10.1098/rspa.2016.0454},
  publisher = {The Royal Society},
}

@Article{Sattinger_1980,
  author    = {Sattinger, D. H.},
  journal   = {Bulletin of the American Mathematical Society},
  title     = {Bifurcation and symmetry breaking in applied mathematics},
  year      = {1980},
  issn      = {1088-9485},
  number    = {2},
  pages     = {779--819},
  volume    = {3},
  doi       = {10.1090/s0273-0979-1980-14823-5},
  publisher = {American Mathematical Society (AMS)},
}

@Article{Zufiria_1987GravityInfD,
  author    = {Zufiria, Juan A.},
  journal   = {Journal of Fluid Mechanics},
  title     = {Non-symmetric gravity waves on water of infinite depth},
  year      = {1987},
  issn      = {1469-7645},
  month     = sep,
  number    = {1},
  pages     = {17},
  volume    = {181},
  doi       = {https://doi.org/10.1017/S002211208700199X},
  publisher = {Cambridge University Press (CUP)},
}

@Article{Zufiria_1987WeaklyGravFinD,
  author    = {Zufiria, Juan A.},
  journal   = {Journal of Fluid Mechanics},
  title     = {Weakly nonlinear non-symmetric gravity waves on water of finite depth},
  year      = {1987},
  issn      = {1469-7645},
  month     = jul,
  number    = {1},
  pages     = {371},
  volume    = {180},
  doi       = {10.1017/s002211208700185x},
  publisher = {Cambridge University Press (CUP)},
}

@Article{Zufiria_1987GravCapFinD,
  author    = {Zufiria, Juan A.},
  journal   = {Journal of Fluid Mechanics},
  title     = {Symmetry breaking in periodic and solitary gravity-capillary waves on water of finite depth},
  year      = {1987},
  issn      = {1469-7645},
  month     = nov,
  pages     = {183--206},
  volume    = {184},
  doi       = {10.1017/s0022112087002854},
  publisher = {Cambridge University Press (CUP)},
}

@Article{Shimizu_2012,
  author    = {Shimizu, Chika and Shōji, Mayumi},
  journal   = {Japan Journal of Industrial and Applied Mathematics},
  title     = {Appearance and disappearance of non-symmetric progressive capillary–gravity waves of deep water},
  year      = {2012},
  issn      = {1868-937X},
  month     = mar,
  number    = {2},
  pages     = {331--353},
  volume    = {29},
  doi       = {10.1007/s13160-012-0070-4},
  publisher = {Springer Science and Business Media LLC},
}

@Book{Okamoto_2001,
  author    = {Okamoto, Hisashi and Shõji, Mayumi},
  publisher = {World Scientific},
  title     = {The Mathematical Theory of Permanent Progressive Water-Waves},
  year      = {2001},
  isbn      = {9789812810441},
  month     = sep,
  doi       = {10.1142/4547},
  issn      = {1793-1045},
  journal   = {Advanced Series in Nonlinear Dynamics},
}

@Article{Chen_1980,
  author    = {Chen, B. and Saffman, P. G.},
  journal   = {Studies in Applied Mathematics},
  title     = {Numerical Evidence for the Existence of New Types of Gravity Waves of Permanent Form on Deep Water},
  year      = {1980},
  issn      = {1467-9590},
  month     = feb,
  number    = {1},
  pages     = {1--21},
  volume    = {62},
  doi       = {10.1002/sapm19806211},
  publisher = {Wiley},
}

@Article{Craig_1988,
  author    = {Craig, Walter and Sternberg, Peter},
  journal   = {Communications in Partial Differential Equations},
  title     = {Symmetry of solitary waves: Solitary waves},
  year      = {1988},
  issn      = {1532-4133},
  month     = jan,
  number    = {5},
  pages     = {603--633},
  volume    = {13},
  doi       = {10.1080/03605308808820554},
  publisher = {Informa UK Limited},
}

@Article{Crandall_1971,
  author    = {Crandall, Michael G and Rabinowitz, Paul H},
  journal   = {Journal of Functional Analysis},
  title     = {Bifurcation from simple eigenvalues},
  year      = {1971},
  issn      = {0022-1236},
  month     = oct,
  number    = {2},
  pages     = {321--340},
  volume    = {8},
  doi       = {10.1016/0022-1236(71)90015-2},
  publisher = {Elsevier BV},
}

@Article{Lannes_2013,
  author    = {Lannes, David and Saut, Jean-Claude},
  journal   = {Kinetic and Related Models},
  title     = {Remarks on the full dispersion Kadomtsev-Petviashvli equation},
  year      = {2013},
  issn      = {1937-5077},
  number    = {4},
  pages     = {989--1009},
  volume    = {6},
  doi       = {10.3934/krm.2013.6.989},
  publisher = {American Institute of Mathematical Sciences (AIMS)},
}

@Article{Whitham1967,
  author    = {Gerald Beresford Whitham},
  journal   = {Proceedings of the Royal Society of London. Series A. Mathematical and Physical Sciences},
  title     = {Variational methods and applications to water waves},
  year      = {1967},
  issn      = {2053-9169},
  month     = jun,
  number    = {1456},
  pages     = {6--25},
  volume    = {299},
  doi       = {10.1098/rspa.1967.0119},
  publisher = {The Royal Society},
}

@Book{Taylor_1996,
  author    = {Taylor, Michael E.},
  publisher = {Springer New York},
  title     = {Partial Differential Equations III},
  year      = {1996},
  isbn      = {9781475741902},
  doi       = {10.1007/978-1-4757-4190-2},
  issn      = {0066-5452},
  journal   = {Applied Mathematical Sciences},
}

@Article{Haziot_2022,
  author    = {Haziot, Susanna and Hur, Vera and Strauss, Walter and Toland, J. and Wahlén, Erik and Walsh, Samuel and Wheeler, Miles},
  journal   = {Quarterly of Applied Mathematics},
  title     = {Traveling water waves — the ebb and flow of two centuries},
  year      = {2022},
  issn      = {1552-4485},
  month     = mar,
  number    = {2},
  pages     = {317--401},
  volume    = {80},
  doi       = {10.1090/qam/1614},
  publisher = {American Mathematical Society (AMS)},
}

@Book{Lannes2013,
  author    = {Lannes, David},
  publisher = {American Mathematical Society},
  title     = {The Water Waves Problem},
  year      = {2013},
  isbn      = {9781470409487},
  month     = may,
  doi       = {10.1090/surv/188},
  issn      = {2331-7159},
  journal   = {Mathematical Surveys and Monographs},
}

@Article{Ehrnstrom_2019,
  author    = {Ehrnström, Mats and Johnson, Mathew A. and Maehlen, Ola I. H. and Remonato, Filippo},
  journal   = {Water Waves},
  title     = {On the Bifurcation Diagram of the Capillary–Gravity Whitham Equation},
  year      = {2019},
  issn      = {2523-3688},
  month     = nov,
  number    = {2},
  pages     = {275--313},
  volume    = {1},
  doi       = {10.1007/s42286-019-00019-4},
  publisher = {Springer Science and Business Media LLC},
}

@Article{Ehrnstrom_2009,
  author    = {Ehrnstrom, M. and Holden, H. and Raynaud, X.},
  journal   = {International Mathematics Research Notices},
  title     = {Symmetric Waves Are Traveling Waves},
  year      = {2009},
  issn      = {1687-0247},
  month     = jul,
  doi       = {10.1093/imrn/rnp100},
  publisher = {Oxford University Press (OUP)},
}

@Article{Constantin2004,
  author    = {Constantin, Adrian and Escher, Joachim},
  journal   = {Journal of Fluid Mechanics},
  title     = {Symmetry of steady periodic surface water waves with vorticity},
  year      = {2004},
  issn      = {1469-7645},
  month     = jan,
  pages     = {171--181},
  volume    = {498},
  doi       = {10.1017/s0022112003006773},
  publisher = {Cambridge University Press (CUP)},
}
\end{document}